\documentclass[12pt]{amsart}
\usepackage[foot]{amsaddr}
\usepackage{amssymb,amscd,amsthm,verbatim,amsmath,color,fancyhdr,mathrsfs}
\usepackage{mathtools}
\usepackage{bbm}
\usepackage{hyperref}
\hypersetup{colorlinks,allcolors=blue}

\usepackage[letterpaper, left=2.5cm, right=2.5cm, top=2.5cm,bottom=2.5cm,dvips]{geometry}

\DeclarePairedDelimiterX{\bracket}[3]{#1}{#2}{#3}

\newcommand{\abs}[1]{\bracket*{\lvert}{\rvert}{#1}}

\DeclarePairedDelimiterXPP\prob[1]{\mathbb{P}}{\lbrace}{\rbrace}{}{#1} 

\DeclarePairedDelimiterXPP\probability[2]{\mathbb{P}_{#1}}{\lbrace}{\rbrace}{}{#2} 

\DeclarePairedDelimiterXPP\expectation[1]{\mathbb{E}}{\lbrack}{\rbrack}{}{#1} 
\newcommand{\E}[1]{\expectation*{#1}} 

\DeclarePairedDelimiterXPP\expectationdist[2]{\mathbb{E}_{#1}}{\lbrack}{\rbrack}{}{#2} 

\DeclarePairedDelimiterXPP\variance[1]{\mathrm{Var}}{\lbrack}{\rbrack}{}{#1} 
\newcommand{\Var}[1]{\variance*{#1}} 

\DeclarePairedDelimiterXPP\variancedist[2]{\mathrm{Var}_{#1}}{\lbrack}{\rbrack}{}{#2} 

\DeclarePairedDelimiterXPP\covariance[2]{\mathrm{Cov}}{(}{)}{}{#1,\mathopen{}#2} 

\DeclarePairedDelimiterX\Set[1]\{\}{
	
	#1
}

\makeatletter
\if@titlepage
  \newenvironment{acknowledgement}{%
    \titlepage
    \null\vfil
    \@beginparpenalty\@lowpenalty
    \begin{center}%
      \bfseries \ackname
      \@endparpenalty\@M
    \end{center}}%
  {\par\vfil\null\endtitlepage}
\else
  
\fi
\makeatother

\newtheorem{theorem}{Theorem}[section]

\newtheorem{lemma}[theorem]{Lemma}
\newtheorem{remark}[theorem]{Remark}

\newtheorem{example}[theorem]{Example}

\title{Central limit theorem in complete feedback games}
\author{Andrea Ottolini}
\address[]{Department of Mathematics, University of Washington, Seattle, WA 98195, USA}
\email{ottolini@uw.edu}
\author{Raghavendra Tripathi}
\address[]{Department of Mathematics, University of Washington, Seattle, WA 98195, USA}
\email{raghavt@uw.edu}

\begin{document}

\begin{abstract}
    Consider a well-shuffled deck of cards of $n$ different types where each type occurs $m$ times. In a complete feedback game, a player is asked to guess the top card from the deck. After each guess, the top card is revealed to the player and is removed from the deck. The total number of correct guesses in a complete feedback game has attracted significant interest in last few decades. Under different regimes of $m, n$, the expected number of correct guesses, under the greedy (optimal) strategy, has been obtained by various authors, while there are not many results available about the fluctuations. In this paper, we establish a central limit theorem with Berry-Esseen bounds when $m$ is fixed and $n$ is large. Our results extend to the case of decks where different types may have different multiplicity, under suitable assumptions.
\end{abstract}
\maketitle
\section{Introduction}\label{sec:Intro}
\subsection{Lady tasting tea, revisited.} Muriel Bristol, a biologist at Rothamsted Research at the dawn of last century, once claimed that she could taste whether a cup of tea was prepared by pouring milk first. Ronald Fisher, in an attempt to disprove her claim, arranged the following simple experiment: Bristol was presented with eight cups of tea, half of which prepared by pouring milk first, and she was asked to taste them one by one and identify the correct ones. This episode became widely known as the ``lady tasting tea" experiment, the very first example appearing in Fisher's seminal book~\cite{fisher1936design} on the design of statistical experiments. Its analysis of the experiment is now a cornerstone of scientific thinking, being also the first appearance of the expression ``null-hypothesis" in Fisher's work. In this case, it assumes that Bristol's guesses are random, so that that the distribution of the score follows an hypergeometric distribution. In average, one expects four correct guesses, the underlying probability distribution being well-understood.
\\ \\
In the original experiment, Bristol did not receive any kind of feedback during the experiment: what if she was told the correct answer after each attempt? Clearly, since she knows the exact number of cups of each type, she can always guess the one that appeared the smallest number of times so far, and therefore increase her likelihood of a correct guess at each step. This does not require any special ability on her side, other than a clever exploitation of the information she is provided with. There has been a substantial flurry of interest in variations of this kind, owing to the connection with randomized clinical trials~\cite{Blackwell1957} and the testing of claims of extra-sensory perceptions~\cite{diaconis1978statistical}, which we will review later. While most of the focus has been on the asymptotic expected score for large experiments, it is clear that a rigorous analysis of the experiments requires the understanding of the fluctuations of the score. This is the focus of our paper. 
\subsection{Model and main result}\label{subs:MainResult}
Consider a well-shuffled deck of card consisting of $n$ different types of cards where each card appears $m$ times. Thus there are $mn$ number of cards in total. Consider the following \emph{complete feedback game}: a player is asked to guess the type of the card appearing on top of the deck. After each guess the top card is revealed to the player. The game continues until the deck is exhausted. Let $S_{m,n}$ denote the total number of correct guess (also referred to as the score) at the end of the game. Obviously the score depends on the strategy. For instance, if the player keeps guessing, say card of type $1$, then $S_{m, n}=m$.  It is shown in~\cite{DG81} that the the greedy algorithm maximizes the expected number of correct guesses, that is to say, a player should guess a card that has the maximum multiplicity in the remaining deck. We will refer to the greedy algorithm as the `optimal strategy' throughout this paper, and we will tacitly assume that the player is performing this strategy. Our main result is a central limit theorem (CLT) for the optimal score $S_{m, n}$ with a Berry-Esseen bound that can be stated as follows. 
\begin{theorem}
\label{thm:CLT_prelim}
Consider a deck of card with $n$ distinct type of cards where each card appears with a fixed multiplicity $m$. Let $S_{n}$ be the total number of correct guesses under the greedy/optimal strategy. Then:
\begin{itemize}
    \item The mean $\mu_{n}:=\E{S_{n}}$ and the variance $\sigma_{n}^2\coloneqq \Var{S_{n}}$ satisfy 
    \begin{align*}
        \mu_{n}\sim \sigma^2_{n}\sim \left(1+\frac{1}{2}+\ldots+\frac{1}{m}\right)\ln n.
    \end{align*}
    as $n\rightarrow +\infty$.
    \item There exists a constant $C=C(m)$ depending only on $m$ such that
    \begin{align*}
    \sup_{x\in\mathbb R}\left|\mathbb P\left(\frac{S_{n}-\mu_n}{\sigma_n}\leq x\right)-\Phi(x)\right|\leq C\,\frac{\ln\ln n}{\sqrt{\ln n}}
    \end{align*}
where $\Phi$ is the cumulative distribution function of a standard normal random variable.
\end{itemize}
\end{theorem}

More generally, we prove a CLT result analogous to Theorem~\ref{thm:CLT_prelim} for deck of cards where the cards of each type occur with (possibly) different multiplicity. That is, for a deck with $n$ different types of cards where the cards of type $i$ for $i\in [n]$ appears $m_i$ times.  We always assume that, for all $n$, the deck is shuffled so that all arrangements of the decks are equally likely. 
\paragraph{Notations}
To state our theorem clearly, we need to fix some notations. 
\begin{enumerate}
    \item For each $n\in \mathbb{N}$, a vector ${\bf m}:= {\bf m}^{n}=(m_1, \ldots, m_n)$ denotes a deck of cards with $n$ different types of card and where a card of type $i\in [n]$ appears $m_i$ times in the deck. 
    \item Denote by $\abs{\bf{m}}=\abs{{\bf m}^{n}}$ the total number of cards in the deck, that is, $\abs{\bf{m}}=\sum_{i=1}^{n}m_i$.
    \item Let ${\bf m}^{n}_{\mathrm{max}}$ denote the highest multiplicity of a card in the deck. That is, ${\bf m}^{n}_{\mathrm{max}}=\max_{i\in [n]} m_i$.
    \item Let $\epsilon_{n}$ denote the fraction of type(s) $i$ such that the cards of type $i$ occur with highest multiplicity ${\bf m}^{n}_{\mathrm{max}}$.
    \item Let $S_{{\bf m}^{n}}$ be the total number of correct guesses (also referred to as the score) at the end of the game under the optimal strategy.
\end{enumerate}


\begin{theorem}
\label{thm:CLT}
Let ${\bf m}^{n}$ be a sequence of decks, indexed by $n$, with $n$-distinct types of cards. Suppose that ${\bf m}^{n}_{\mathrm{max}}\leq m$ for some $m$ that is independent of $n$, and that $\epsilon_{n}\geq \epsilon$ for some positive $\epsilon$ independent of $n$. Let $S_{{\bf m}^n}$ be the total number of correct guesses under the greedy/optimal strategy. Then:
\begin{itemize}
    \item The mean $\mu_n:=\E{S_{{\bf m}^{n}}}$ and the variance $\sigma_n^2\coloneqq \Var{S_{{\bf m}^{n}}}$ satisfy 
    \begin{align*}
        \mu_n\sim \sigma^2_n\sim \left(1+\frac{1}{2}+\ldots+\frac{1}{{\bf m}^{n}_{\mathrm{max}}}\right)\ln n.
    \end{align*}
    as $n\rightarrow +\infty$.
    \item There exists a constant $C=C(\epsilon, m)$ such that
    \begin{align*}
    \sup_{x\in\mathbb R}\left|\mathbb P\left(\frac{S_{{\bf m}^{n}}-\mu_n}{\sigma_n}\leq x\right)-\Phi(x)\right|\leq C\,\frac{\ln\ln n}{\sqrt{\ln n}}
    \end{align*}
where $\Phi$ is the cumulative distribution function of a standard normal random variable.
\end{itemize}
\end{theorem}
\begin{remark}
It is clear that all the assumptions in the Theorem~\ref{thm:CLT} are satisfied if $m_i=m$ for all $i\in [n]$. In particular,  Theorem~\ref{thm:CLT_prelim} follows from Theorem~\ref{thm:CLT}. The result about the mean was already shown in~\cite{DG81, he2021card}.
\end{remark}

\subsection{Related literature}\label{subsec:Lit}
The complete feedback game was originally motivated by clinical trials. For an in-depth discussion about the problem, a good reference is~\cite{efron1971forcing}, though the first appearance is in a work by Blackwell and Hodges~\cite{Blackwell1957}. They considered the case where two types of treatments have to be assigned to a fixed number of people, say $2m$, who arrive one by one at the clinic. They were interested in the case where both treatments are provided in the same quantity and in a random order. However, they assume that the hospital may decide, to their discretion, whether to rule out some of the subjects because of their medical conditions. Since they have information on the treatments provided up to that point, they may decide to bias the result of the experiment toward a specific treatment. This can be done by ruling out a particularly sick subject if they know that it is more likely that their favorable treatment has to appear next. \\ \\
In our language, this is precisely the complete feedback case with $n=2$, with ${\bf m} =(m,m)$. The authors in~\cite{Blackwell1957} gave an asymptotic formula for the optimal expected score (in their language, the selection bias), which was then extended by~\cite{DG81} to the generic case ${\bf m}=(m_1, \ldots, m_n)$ with $n$ fixed. As for the fluctuations, the latter reference shows that, for $n=2$ and the ${\bf m}=(m,m)$ with $m$ large, the limiting optimal score satisfies a central limit theorem. On the other hand, in the unbalanced case where ${\bf m}=(m_1, m_2)$ where $m_1, m_2$ grow with $m_1/m_2\rightarrow p\neq 1/2$, they show that the fluctuations of the optimal score are \emph{not Gaussian}. Related results in the case $m=2$ also appeared in~\cite{kuba2023card}.\\ \\
Another occurrence of the complete feedback game is related to the rigorous analysis for extra-sensory perception claims. In fact, one of the most celebrated experiment in this direction corresponds precisely to the complete feedback game with a deck of twenty-five cards (Zener cards), with five symbols each appearing five times. For an historical account, the interested reader is referred to~\cite{diaconis1978statistical}. Motivated by this, in~\cite{DG81} the authors suggest to study the asymptotic optimal expected score for the complete feedback game with decks ${\bf m}=(m_1, \ldots, m_n)$ where $n$ grows. In the case $m_i\equiv 1$, the analysis becomes much simpler since the sequence of guesses become independent. In particular, it is easy to deduce that one obtains about $\ln n$ correct guesses in expectation, with a variance of the same order and normal fluctuations.\\ \\
The case where some of the $m_i$ are greater than one is more subtle, since the chance of a correct guess will depend on the history of the draws up to that moment. In~\cite{diaconis_partialfeedback}, the authors analyzed the case where $m_i\equiv m$ is fixed and $n$ grows to infinity, showing that asymptotically the expected optimal score is $\left(1+\ldots+ 1/m\right)\ln n$. The result was substantially refined by the first named author and He in~\cite{he2021card}, where the expected score is determined for decks ${\bf m}=(m_1,\ldots, m_n)$ under the same assumptions of Theorem~\ref{thm:CLT}. Moreover, their asymptotic result matches the optimal expected score up to an explicit error that goes to zero. Their main tool is the analysis of a certain variation of the birthday problem via Stein's methods, which will be our main tool here as well. In the case $m_i\equiv m$ where both $m$ and $n$ are growing, the asymptotic for the expected optimal score was obtained by the first named author and Steinerberger in~\cite{ottolini2022guessing}, covering a variety of regimes that include the case $n=m$ (the original Zener's setting).
\\ \\
Variations of the game also include different types of feedback. The most relevant case being that of a yes/no feedback (i.e., the card is shown only when a guess is correct). This becomes much harder to analyze even for balanced decks ${\bf m}=(m_1,\ldots, m_n)$ with $m_i\equiv m$.  It is known~\cite{DG81} that the optimal strategy is \emph{not the greedy one} as soon as $m>1$ and $n>2$. Some limiting results were recently obtained in~\cite{diaconis_partialfeedback, nie2022number}. For instance, that the expected optimal score for $n\gg m \gg 1$ is of the form $m+\Theta(\sqrt m)$ uniformly in $n$. Results on the fluctuations are currently unknown -- except in the case $m=1$ -- where the limiting distribution has a non-normal behavior as shown in~\cite{DG81}. Since the optimal strategy is rather hard to implement, a fact ultimately due to its connection with permanents~\cite{chung1981permanents, diaconis2001statistical}, there has also been some interest in near-optimal strategies that are easier to implement~\cite{diaconis2022guessing}. 
\\ \\
Finally, we mention some other variations of the game on a similar flavour. The problem of minimizing the expected number of correct guesses was also addressed in~\cite{DG81, he2021card, diaconis_partialfeedback}-- for both the complete feedback and the yes/no feedback. Another natural set of questions comes from considering decks that have not been properly shuffled, such as the case of a deck which has been riffle shuffled~\cite{liu2021card}.

\section{Proofs}\label{sec:proof}
In the remainder of the paper, we will drop at times the dependence on the deck $\bf m$ and on $n$. The implicit constants in the notation $O, \Omega, \Theta, \lesssim$ will depend on $m$ and $\epsilon$ only, unless we specify otherwise. We will often identify ${\bf m}^{n}_{\mathrm{max}}$ and $\epsilon_{n}$ with their upper/lower bound $m$ or $\epsilon$, unless there is ambiguity.
\subsection{Main idea}
It will be convenient to define the following random variables and setup some notations.
\begin{itemize}
    \item $T_j=\max\left\{t\in \{0,1,
    \ldots, \abs{\bf{m}}\right\}:\text{No card among the \emph{last} $t$ appear more than $j$ times}\}$. Here $0\leq j\leq m$, with the convention $T_0=0$ and $T_m=\abs{\bf{m}}$. 
    \item $W_{j,t}=\sum_{{\bf t}\leq t} Y_{\bf t}$, where for each index $\bf t=(t_1,\ldots, t_{j+1})$ the notation ${\bf t}\leq t$ means ${\bf t}_{\ell}\leq t$ for each $\ell=1, 2,\ldots, j+1$ and the binary random variable $Y_{\bf t}$ is one if and only if the cards at positions $\bf t$ are equal (here, positions are considered from the \emph{bottom} of the deck). Notice that $T_j>t$ if and only if $W_{j,t}=0$. Here, $1\leq t\leq \abs{\bf{m}}$ and $0\leq j\leq m$.
    \item $\widetilde W_j=W_{j-1, T_j}$ denotes the number of cards that appear $j$ times before some card appear $j+1$ times. Again, this is done from the \emph{bottom} of the deck. Here, $1\leq j\leq m$.
\end{itemize}
\begin{example}
Consider a deck ${\bf m}=(3,3,2)$, and assume that the sequence of cards extracted, \emph{listed from the last to the first}, is 
\begin{align*}
    (1, 2, 2, 1, 3, 1, 2, 3).
\end{align*}
In this case, $T_0=0,\, T_1=2, \,T_2=5,\, T_3=8$. Correspondingly, we have $\widetilde W_1=W_{0, 2}=2,\, \widetilde W_2=W_{1, 5}=2,\, \widetilde W_3=W_{2, 8}=2$. For instance, $\widetilde W_2=2$ reflects the fact that there are two couple of identical cards among the last five and no triple of identical cards (those labeled one and two), while among the last six cards there is a triple of identical cards (those labeled one).  
\end{example}
\begin{remark}\label{rem:mult}
Notice that $\widetilde W_m$ is equal to the number of types that appear with multiplicity $m$, which is at least $\epsilon n$ under the assumption of Theorem~\ref{thm:CLT}.
\end{remark}
\begin{remark}\label{rem: atleastpolynomial}
Since $W_{j,t}$ is a sum of indicators, $W_{j,t}\leq W_{j, \abs{\bf{m}}}\leq \binom{nm}{j}$ is at most polynomial in $n$ under the assumption of Theorem~\ref{thm:CLT}.
\end{remark}
The random variables $W_{j, t}$ are important tools in understanding the asymptotic behaviour of total score. In fact, the main ingredient in~\cite{he2021card} is an asymptotic result for the $W_{t, j}$, which behave like Poisson random variables with suitable parameters. This should come as no surprise, since the $Y_{\bf t}$s are indicator of rare events, most of which are weakly dependent. They obtain the following result.
\begin{theorem}[Theorem 1.8~\cite{he2021card}]
\label{thm:poissonapp}
Let $1\leq j<m$. Then, there exists $\lambda=\Theta\left(\frac{t^{j+1}}{n^j}\right)$ such that
\begin{align*}
    d_{tv}(W_{j,t}, Poi(\lambda))\lesssim \frac{t}{n}\lesssim \left(\frac{\lambda}{n}\right)^{\frac{1}{j+1}}
\end{align*}
Here, $Poi(\lambda)$ represents a Poisson random variable with mean $\lambda$, and $d_{tv}$ represents the total variation distance between probability measures (where we identify a random variable with its law).
Moreover, the implicit constants in the error term and the definition of $\lambda$ can be chosen to depend only on $j$ and $\epsilon$, the fraction of types that appear with multiplicity $m$. 
\end{theorem}
\begin{remark}
In the case $m_i\equiv m$, one has $\lambda=\frac{t^{j+1}}{n^j}\frac{\binom{m}{j+1}}{m^{j+1}}$. Notice that, for fixed $j$, the second term converges to $\frac{1}{(j+1)!}$ as $m\rightarrow +\infty$.
\end{remark}
\begin{remark}
Theorem~\ref{thm:poissonapp} can be thought as a variant of the classical birthday problem. In fact, for $j=1$ and $m_i\equiv {\bf m}^{n}_{\mathrm{max}}\rightarrow +\infty$, we obtain the best possible approximation for the classical birthday problem.
\end{remark}
In this section we state and prove some lemmas that will be useful in the proof of Theorem~\ref{thm:CLT}. We begin with a discussion of the idea of the proof. The total score $S_{\bf m}$ can be written as the sum of $\abs{\bf{m}}$ (the total size of the deck) binary random variables, namely, the indicators that at a given time the player obtains a correct guess. Were these random variables independent, the CLT for the total score would follow at once. \\ 

However, even in the case $m=2, n=2$, it is clear that if the second guess is correct, then the remaining two cards are distinct (hence the third guess will be correct with probability $1/2$) while if the second guess is wrong, then the remaining two cards are the same. Hence the third guess will be correct with probability $1$. Intuitively, the dependence becomes weak for large $n$, and it should be related to the concentration properties of the random variables introduced above. Indeed, the strategy will change depending on how many card appear with a given multiplicity at a given time. The first crucial step towards the proof of CLT will be the observation that, conditioned on $\widetilde{W}_j$, the score can be written as a sum of independent random variables (see Lemma~\ref{lem:TotalScore}). This allows us to prove a CLT for the conditional score with a uniform Berry-Essen bound (Lemma~\ref{lem:berryEssen}). 

The final issue is to understand the behavior of $\widetilde W_j$s and show that they enjoy suitable concentration. The main difficulty is that Theorem~\ref{thm:poissonapp} requires a \emph{fixed} time, rather than the random time $T_j$ appearing in the definition of the $\widetilde W_j$. Moreover, the random variables $T_j$ and $W_{j-1, t}$ are dependent meaning that we cannot expect a straightforward limiting result expressed in terms of a compound Poisson random variables. It is worth noticing that, while in~\cite{he2021card} there is a multivariate version of Theorem~\ref{thm:poissonapp}, it does not suffice for our purposes. To circumvent the problem, we will use a suitable monotonicity and concentration argument, which allows to prove the main Lemma~\ref{lem:finalReduction}.
\subsection{A CLT for the conditional score}
We start by showing a useful representation of the optimal score. 
\begin{lemma}
\label{lem:TotalScore}
For any deck ${\bf m}$, the optimal score $S_{\bf m}$ can be written as
\begin{align*}
    S_{\mathbf{m}}=\sum_{j=1}^{m}\sum_{s=1}^{\widetilde W_j}X_{j,s}, 
\end{align*}
where the $X_{j,s}$ are conditionally independent -- given the $\widetilde W_j$'s -- Bernoulli random variables with $\mathbb P(X_{j,s}=1)=\frac{1}{s}$. 
\end{lemma}
\begin{proof}
Let $\tau_{j, s}$ be the time at which the maximum multiplicity of a symbol left in the deck is equal to $j$, there are exactly $s$ symbols with this multiplicity, and the symbol of the card on top of the deck is precisely one of those $s$ symbols. Notice that a correct guess can only occur at the times $\tau_{j, s}$. If $X_{j, s}$ denotes the indicator that the guess at time $\tau_{j, s}$ is correct, then
\begin{align*}
\mathbb P\left(X_{j,s}=1\right)=\frac{1}{s}.
\end{align*}
To see this, notice that at time $\tau_{j,s}$, the symbol appearing on the card is one of the $s$ symbols that appear with the maximum multiplicity, and each of them appear with the same likelihood since the deck is uniformly shuffled. It is worth remarking that, while the optimal strategy is not unique (the player is free to choose any of the $s$ symbols that appear with the maximum multiplicity: he/she may always guess, e.g., the symbol they like the most among those $s$, or a uniformly random among those), the distribution of the $X_{j,s}$ remain uniform. Moreover, these random variables are conditionally independent, given the times $\tau_{s,j}$. Finally, observe that for each $j$, the number of relevant $\tau_{j,s}$ is $1\leq j\leq m$ and $1\leq s\leq \widetilde W_{j}$, so that we obtain conditional independence given the $\widetilde W_j$s. 
\end{proof}

\begin{remark}
Following Remark~\ref{rem:mult}, our assumption on $\epsilon$ guarantees that the expected number of correct guesses is lower bounded by $\ln n+O(1)$. This can be seen by looking at $\widetilde W_m$ (i.e., guesses early on in the game). 
\end{remark}
\begin{remark}
Since $\widetilde W_{j}\geq 1$ for all $j$, we have the deterministic bound $S_{\bf m}\geq m$. These correspond to the correct guesses when there is only one card appearing with the maximum multiplicity, which eventually will result in a correct guess with certainty. 
\end{remark}
For convenience, we will denote by $S_{\bf m}'$ the total score conditioned on the random variables $\widetilde{W}_j$ for $1\leq j\leq m$. Lemma~\ref{lem:TotalScore} says that $S_{\bf m}'$ is a sum of independent Bernoulli random variables. We can leverage this to obtain the following result. 
\begin{lemma}
\label{lem:berryEssen}
Consider a deck as in the assumption of Theorem~\ref{thm:CLT}. Let $S_{\bf m}'$ denote the total score conditioned on the $\widetilde W_j$s, and let $\mu'_n$ and $\sigma'_n$ denote the conditional mean and standard deviation. Then, 
\begin{equation*}
    \mu'_n=\sum_{j=1}^m \ln \widetilde W_j+O(1),\quad  \sigma'^2_n=\mu'_n+O(1).
\end{equation*}

Moreover, uniformly over $\widetilde W_j$s, one has
\begin{align*}
    \left|\mathbb P\left(\frac{S'_{\bf m}-\mu'_n}{\sigma'_n}\leq x\right)-\Phi(x)\right|\lesssim \frac{1}{(\ln n)^{\frac{3}{2}}}
\end{align*}
where $\Phi(x)$ denotes the CDF of a standard normal random variable. 
\end{lemma}
\begin{proof}
By means of Lemma~\ref{lem:TotalScore} and linearity of expectation, we can write
\begin{align*}
&\mu_n'=\sum_{j=1}^m\left(1+\frac{1}{2}+\ldots+\frac{1}{\widetilde W_j}\right), \\ &\sigma_n'^2=\sum_{j=1}^m\left[\left(1+\frac{1}{2}+\ldots+\frac{1}{\widetilde W_j}\right)-\left(1+\frac{1}{2^2}+\ldots+\frac{1}{
\widetilde W_j^2}\right)\right].
\end{align*}
Using the well-known facts 
\begin{align*}
    1+\frac{1}{2}+\ldots+\frac{1}{k}=\ln k+O(1), \quad \sum_{n=1}^{+\infty}\frac{1}{k^2}=\frac{\pi^2}{6}<\infty
\end{align*}
we deduce that
\begin{equation}\label{eq:condmean}
    \mu'_n=\sum_{j=1}^m \ln \widetilde W_j+O(1)
    ,\quad  
    \sigma'^2_n=\mu'_n+O(1).
\end{equation}
Moreover, using the fact that $\widetilde W_{m}\geq \epsilon n$ (see Remark~\ref{rem:mult}) we deduce that
\begin{equation}\label{eq:useful}
    {\sigma'_n}^2\geq \ln n+O(1)
\end{equation}
uniformly over all realizations of the $\widetilde W_j$s. Since the second and third moments of a Bernoulli random variable are the same, a standard Berry-Esseen bound for non identically distributed random variables (see, e.g.,~\cite{shevtsova2010improvement}) gives  
\begin{align*}
    \left|\mathbb P\left(\frac{S'_{\bf m}-\mu'_n}{\sigma'_n}\leq x\right)-\Phi(x)\right|\lesssim \frac{1}{\sigma'^3_n}\lesssim \frac{1}{(\ln n)^{\frac{3}{2}}}
\end{align*}
\end{proof}
\begin{remark}\label{rem:mixture}
Notice that the proof already shows that the limiting fluctuations for the score $S_{{\bf m}^{n}}$ are distributed as mixtures of normal random variables. In order to prove Theorem~\ref{thm:CLT}, it will suffice to show suitable concentration for the conditional mean and variance. 
\end{remark}

\subsection{Removing the conditioning}
As pointed out in Remark~\ref{rem:mixture}, the conditional CLT for $S_{\bf m}'$ shown in Lemma~\ref{lem:berryEssen} will suffice to our purposes if we can show a suitable concentration for the conditional means and variance $\mu_n'$ and $\sigma_n'$. Towards this, the main ingredient will be the following. 
\begin{lemma}\label{lem:finalReduction}
Consider a deck ${\bf m}$ satisfying the assumptions of Theorem~\ref{thm:CLT}. Let $\mu_n'$ be the conditional mean of the total score $S_{\bf m}$ given the random variables $\widetilde{W}_j$s. Then, 
\[\Var{\mu_n'}=O\left(\left(\ln \ln n\right)^2\right)\;.\]
\end{lemma}
\begin{proof}
First, we claim that it suffices to show
\begin{equation}\label{eq:goal}
    \mathbb E\left[\left(\ln \widetilde W_j-c_j\right)^2\right]=O\left((\ln\ln n)^2\right)
\end{equation}
where $c_j=c_j(n)=\frac{\ln n}{j+1}$, and $1\leq j\leq m-1$ (the case $j=m$ is obvious since $\widetilde W_m$ is deterministic). Here, the implicit constant in $O((\ln\ln n)^2)$ depends on the maximum multiplicity of the deck ${\bf m}_{\max}$ which is bounded by some constant $m$. Therefore assuming~\eqref{eq:goal}, the desired conclusion follows, using the first bound in~\eqref{eq:condmean} together with the triangle inequality and the well-known fact that the variance minimizes the square discrepancy from any constant.

The goal now is to reduce concentration properties of $\widetilde W_j$ to those of $W_{j-1,t}$ and $T_j$, for which we can exploit Theorem~\ref{thm:poissonapp}. To this aim, consider two sequences $f_n=\ln n$ and $g_n=\frac{1}{\ln n}$. Recall that $T_j>t$ if and only if $W_{j, t}=0$. In  particular, Theorem~\ref{thm:poissonapp} allows us to approximate probabilities of the form $T_j\in [a,b]$, with an explicit error term. More precisely, we have
\begin{align}\label{eqn:ErrorBound1}
    \left|\mathbb P\left(\frac{T_j}{Cn^{\frac{j}{j+1}}}\not\in [g_n, f_n]\right)-\left(e^{-f_n}+1-e^{-g_n}\right)\right|\lesssim \frac{f_n}{n^{\frac{1}{j+1}}}\lesssim \frac{\ln n}{n^{\frac{1}{j+1}}}\;.
\end{align}
Let $t_{-}:=Cn^{\frac{j}{j+1}}g_n$ and let $t_{+}:=Cn^{\frac{j}{j+1}}f_n$. Using the fact that $1\leq \widetilde W_j\leq n$, as well as the definition of $c_j$, we have $\ln \widetilde W_j, c_j\leq  \ln n$. Therefore, on the event $T_{j}\not\in [t_{-}, t_{+}]$ we obtain 
\[\mathbb E\left[\left(\ln \widetilde W_j-c_j\right)^21_{\{T_j\not\in[t_{-}, t_{+}]\}}\right]=O\left((\ln\ln n)^2\right)\;.\]
We can thus restrict our attention to the event $T_j\in [t_{-}, t_{+}]$. Since $W_{j-1, t}$ is weakly increasing in $t$, in the regime $T_j\in [t_{-}, t_{+}]$, one has
\begin{align*}
  W_{j-1, t_{-}}\leq  \widetilde W_j\leq W_{j-1, t_{+}}.
\end{align*}
The key gain is that we overcome the complicate dependence mechanism behind the definition of the $\widetilde W_j$ -- recall that $T_j$ and $W_{j-1,t}$ are \emph{not} independent. Using Theorem~\ref{thm:poissonapp}, we know 
\begin{align*}
    d_{tv}\left(W_{j-1, t_{-}}, Poi\left(\lambda_{-}\right)\right)\lesssim\frac{t_{-}}{n}\lesssim \frac{g_n}{n^{\frac{1}{j+1}}}, \quad d_{tv}\left(W_{j-1, t_{+}}, Poi\left(\lambda_{+}\right)\right)\lesssim \frac{t_{+}}{n}\lesssim\frac{f_n}{n^{\frac{1}{j+1}}}. 
\end{align*}
where $\lambda{+}$ and $\lambda_{-}$ satisfy
\begin{equation}\label{eq:lambda}
\lambda_+=\Theta\left(n^{\frac{1}{j+1}}f_n\right), \quad \lambda_-=\Theta\left(n^{\frac{1}{j+1}}g_n\right).
\end{equation}
In particular, since $g_n=\frac{1}{\ln n}$ we obtain that 
$$
\mathbb P(W_{j-1, t_{-}}=0)\lesssim e^{-\lambda_{-}}+\frac{g_n}{n^{\frac{1}{j+1}}}\lesssim \frac{1}{n^{\frac{1}{j+1}}}\;.
$$
Using the fact that $1\leq \widetilde W_j\leq n$, as well as the definition of $c_j$, we have $\ln \widetilde W_j, c_j\leq  \ln n$. This allows to bound
\begin{align*}
\mathbb E\left[|\ln \widetilde W_j-c_j|^21_{\left\{W_{j-1, t_{-}}=0\right\}}\right]\leq (\ln n)^2 \mathbb P(W_{j-1, t_{-}}=0)\lesssim \frac{(\ln n)^2}{n^{\frac{1}{j+1}}}.
\end{align*}
Therefore, we can restrict our attention to the event $\left\{T_j\in [t_{-}, t_{+}]]\right\}\cap \{W_{j-1, t_{-}}>0\}$. On this event we have
\begin{equation}\label{eq:removetilde}
    \left|\ln \widetilde W_j-c_j\right|^2\leq |\ln W_{j-1, t_{-}}-c_j|^21_{\{W_{j-1, t_{-}}>0\}}+|\ln W_{j-1, t_{+}}-c_j|^21_{\{W_{j-1, t_{+}}>0\}}\;.
\end{equation}
 Using the fact that $c_j=\frac{1}{j+1}\ln n$ and the Remark~\ref{rem: atleastpolynomial}, we have $|\ln W_{j-1, t_{\pm}}-c_j|\lesssim  \ln n$. In particular, we can replace $W_{j-1, t_{\pm}}$ with $Poi(\lambda_{\pm})$ in \eqref{eq:removetilde} up to an error of order
\begin{align*}
  \ln n\left(e^{-f_n}+1-e^{-g_n}+\frac{f_n}{n^{\frac{1}{j+1}}}+\frac{g_n}{n^{\frac{1}{j+1}}}\right).
\end{align*}
Recall that $f_n=\ln n, g_n=\frac{1}{\ln n}$. Therefore, using $1-e^{-x}\leq x$, the above expression above $O(1)$. We are thus left to show that
\begin{align*}
    \mathbb E\left[\left(\ln \left(Poi(\lambda_{+})\right)-c_j\right)^21_{Poi(\lambda_{+})>0}\right]+\mathbb E\left[\left(\ln \left(Poi(\lambda_{-})\right)-c_j\right)^21_{Poi(\lambda_{-})>0}\right]=O\left(\left(\ln\ln n\right)^2\right).
\end{align*}
Thanks to \eqref{eq:lambda} we can replace $c_j$ with $\ln(\lambda_{\pm})$ up to an error. That is,
\begin{align*}
    |c_j-\ln (\lambda_{-})|^2\lesssim (\ln g_n)^2=\left(\ln\ln n\right)^2,
    \quad |c_j-\ln (\lambda_{+})|^2\lesssim (\ln f_n)^2=\left(\ln\ln n\right)^2\;.
\end{align*}
The result, thus, follows by showing
\begin{align*}
    \mathbb E\left[\left(\ln \left(Poi(\lambda_{+})\right)-\ln\lambda_{+}\right)^21_{Poi(\lambda_{+})>0}\right]+\mathbb E\left[\left(\ln \left(Poi(\lambda_{-})\right)-\ln\lambda_{-}\right)^21_{Poi(\lambda_{-})>0}\right]\lesssim \left(\ln\ln n\right)^2.
\end{align*}
In fact, we can do better than this (notice that, because of~\eqref{eq:lambda} and our choices of $f_n$ and $g_n$, we know that $\lambda_{\pm}> 1$ for all $n$ sufficiently large). The proof is complete using~\ref{lem:simple}.
\end{proof}
\begin{lemma}\label{lem:simple}
Let $X=\frac{Poi(\lambda)}{\lambda}$ for $\lambda>1$. Then,
\begin{align*}
    \mathbb E[(\ln X)^21_{X>0}]\leq C
\end{align*}
for some absolute constant $C$.
\end{lemma}
\begin{proof}[Proof of the claim~\ref{lem:simple}]
Let $A=\{|X-1|\geq 1/2, X\neq 0\}$. Then, we have
\begin{align*}
    \mathbb E[(\ln X)^21_{X>0}]\leq (\ln 2)^2+\mathbb E[(\ln X)^21_A].
\end{align*}
Notice that if $X>1$ and we can bound $(\ln X)^2\leq X$. On the other hand, if $\frac{1}{\lambda}\leq X<1$ then $(\ln X)^2\leq (\ln \lambda)^2$. Thus, on the event $A$, we have $(\ln X)^2\leq (\ln \lambda)^2+X$. Therefore, combining with the Cauchy-Schwartz inequality, 
\begin{align*}
    \mathbb E[(\ln X)^21_A]\leq \mathbb{E}[X1_{A}]+(\ln\lambda)^2\mathbb{P}(A)\leq(\mathbb E[X^21_{A}])^{1/2}\left(\mathbb P(A)\right)^{1/2}+(\ln\lambda)^2\mathbb{P}(A).
\end{align*}
For the first term, we note that
\begin{align*}
    \mathbb E[X^2]=\frac{\lambda^2+\lambda}{\lambda^2}\leq 2.
\end{align*}
On the other hand, the Chernoff bound for Poisson random variables gives
\begin{align*}
    \mathbb P(A)\leq \mathbb P\left(|Poi(\lambda)-\lambda|\geq \frac{\lambda}{2} \right)\leq 2e^{-\frac{\lambda}{12}}
\end{align*}
from which the claim follows at once. 
\end{proof}

We are now ready to prove the main result, which will be an easy consequence of the lemmas we proved so far.
\begin{proof}[Proof of Theorem~\ref{thm:CLT}]
We start with the first part. The asymptotic result for the mean is the main result in~\cite{he2021card}. As for the variance $\sigma^2_n$, we can use the law of total variance to write in terms of the conditional mean and variance (see Lemma~\ref{lem:berryEssen}) as
\begin{align*}
    \sigma_n^2=\Var{\mu_n'}+\mathbb E[{\sigma'_n}^2]. 
\end{align*}
Therefore, the asymptotic result follows by using \eqref{eq:condmean}, which shows that the second term is asymptotically the same of $\mu_n$, together with Lemma~\ref{lem:finalReduction}, which shows that the first term is negligible.

We now move to the second part, namely, the proof of the CLT. Let us define a random variable $y = y_n(x)=\frac{\sigma}{\sigma'}x+\frac{\mu-\mu'}{\sigma'}$. Now observe that
\begin{align}
    \left| \mathbb P\left(\frac{S_{{\bf m}^{n}}-\mu_n}{\sigma_n}\leq x\right) - \Phi(x) \right| &= \left| \E{ \mathbb P\left(\frac{S'_{{\bf m}^{n}}-\mu_n}{\sigma_n}\leq x\Bigg|\widetilde W_1, \ldots \widetilde W_m \right) - \Phi(x)} \right|\nonumber\\
    &= \left| \E{ \mathbb P\left(\frac{S'_{{\bf m}^{n}}-\mu'_n}{\sigma'_n}\leq y\Bigg|\widetilde W_1, \ldots \widetilde W_m \right) - \Phi(x)} \right|\nonumber \\
    &\leq \|\widetilde{F}-\Phi\|_{\infty} + \abs{\E{\Phi(y)-\Phi(x)}}\;,
\label{eqn:FirstStep}
\end{align}
where $\widetilde{F}(z):=\mathbb{P}\left(\frac{S_{\bf m}'-\mu'}{\sigma'}\leq z\right)$. We know that $$\|F-\Phi\|_{\infty}\lesssim \frac{1}{\left(\ln n\right)^{\frac{3}{2}}}$$ from Lemma~\ref{lem:berryEssen}. Therefore, it suffices to bound $\abs{\E{\Phi(y)-\Phi(x)}}$.
Using~\eqref{eq:condmean} and \eqref{eq:useful}, we observe that
\begin{align*}
    \sigma_n-\sigma'_n=\frac{\sigma_n^2-{\sigma'_n}^2}{\sigma_n+\sigma_n'}=o\left(|\mu_n-\mu_n'|+1\right)
\end{align*}
with probability one. And, therefore,
\begin{equation}\label{eq:y-x}
    |y(x)-x|\leq \frac{|\sigma_n'-\sigma_n|}{\sigma'_n}|x|+\frac{|\mu_n-\mu'|}{\sigma'_n}\lesssim \left(1+|x|\right)\frac{|\mu_n-\mu_n'|+1}{\sqrt{\ln n}},
\end{equation}
with probability one. 

If $|x|\leq 1$, using $|\Phi(x)-\Phi(y)|\leq |x-y|$ and \eqref{eq:y-x} we conclude 
$$
|\Phi(y)-\Phi(x)|\lesssim \frac{|\mu_n-\mu_n'|+1}{\sqrt{\ln n}}. 
$$
with a constant that does not depend on $x$.  Therefore, Lemma~\ref{lem:finalReduction} allows us to conclude
\begin{align*}
   \mathbb E[|\Phi(y)-\Phi(x)|]\lesssim  \frac{\ln\ln n}{\sqrt{\ln n}}.
\end{align*}

If, instead $|x|\geq 1$, define the event (recall that $y(x)$ is a random variable)
$$A_n=A_n(x)=\{|y(x)-x|\leq |x|/2\}.$$ 
On the event $A_n$, $x$ and $y$ have the same sign, and hence we have the bound
\begin{align*}
    |\Phi(y)-\Phi(x)|\leq |x-y|\max_{z\in [\min(x,y),\max(x,y)]}\Phi'(z)\lesssim |x-y|e^{-\frac{\min(x^2,y^2)}{2}}.
\end{align*}
Therefore, we have with probability one and for all $n$ sufficiently large,
\begin{align*}
    |\Phi(y)-\Phi(x)|\lesssim (1+|x|)e^{-\min(x^2, y^2)}\frac{|\mu_n-\mu_n'|+1}{\sqrt{\ln n}}
\end{align*}
 On this event, the pre-factor depending on $x$ is uniformly bounded from above and thus
\begin{align*}
    |\Phi(y)-\Phi(x)|1_{A_n}\lesssim \frac{|\mu_n-\mu_n'|+1}{\sqrt{\ln n}},
\end{align*}
with the implicit constant in $\lesssim$ independent of $x$. Using Lemma~\ref{lem:finalReduction} together with Cauchy-Schwarz, we conclude that.
\begin{align*}
   \mathbb E[|\Phi(y)-\Phi(x)|1_{A_n}]\lesssim \frac{\ln\ln n}{\sqrt{\ln n}}.
\end{align*}
For $|x|\geq 1$ and on the complement of event $A_n$, we use the bound $|\Phi(x)-\Phi(y)|\leq 1$ to deduce
\begin{align*}
    \mathbb E[|\Phi(y)-\Phi(x)|1_{A^c_n}]\leq \mathbb P(A_n^c).
\end{align*}
However, using \eqref{eq:y-x}, on the complement of $A_n$ we have, 
$$
\frac{|x|}{2}\leq |y(x)-x|\lesssim (1+|x|)\frac{|\mu_n-\mu_n'|+1}{\sqrt{\ln n}}.
$$
In particular, this entails that for some constant $K>0$ (independent of $x$) we have
$$
\frac{|\mu_n-\mu_n'|+1}{\sqrt{\ln n}}\geq K.
$$
We conclude, using Lemma~\ref{lem:finalReduction} once more, together with Chebyshev's inequality, that 
\begin{align*}
    \mathbb P(A_n^c)\leq \mathbb P\left(\frac{|\mu_n-\mu_n'|+1}{\sqrt{\ln n}}\geq K \right)\lesssim \frac{\ln\ln n}{\sqrt{\ln n}}.
\end{align*}
This completes the proof.
\end{proof}

\section{Discussion}\label{sec:Disc}
It is natural to ask whether the assumptions of our main Theorem~\ref{thm:CLT} are needed. The very first obstacle is given by the use of Theorem~\ref{thm:poissonapp}. One could keep track of the dependence of $m$ (resp., $\epsilon$) in their error bounds, and thus extend our result by allowing some moderate growth (resp., decay).

\sloppy It is worth remarking that Lemma~\ref{lem:TotalScore} holds for all decks, while the bound in Lemma~\ref{lem:berryEssen} continues to hold as long as the conditional variance goes to infinity. In particular, under this condition we are guaranteed to have convergence to a mixture of normal random variables. Notice that the fact that $m$ remains bounded plays no role in this part of the proof, while we needed our condition on $\epsilon$. While this may not be sharp, some care has to be taken if one type of card has a much higher multiplicity than all the others: for instance, in the case of finite $n$, this may be an obstacle to the convergence to a mixture of normal (see~\cite{DG81} for the case where $n=2$).

As for the convergence to a normal random variable (i.e., a trivial mixture) our method relies on $m$ being finite, and it is an interesting problem to determine if this is a true limitation. This is intimately related to the understanding of the concentration properties of $\widetilde W_j$, i.e., the ``number of ties at top", when $j$ grows. At least for some regimes of $j$, a closely related result is available in~\cite{ottolini2020oscillations}, where asymptotic for $\widetilde W_j$ are shown if the hypergeometric process is replaced by the multinomial process (i.e., if cards are reinserted in the deck).

It is worth pointing out that, while there is hope for a CLT to hold even when $m$ grows much faster than $n$, the asymptotic of the expected value (first part of Theorem~\ref{thm:CLT}) eventually breaks down, as shown in~\cite{ottolini2022guessing}.

\section*{Acknowledgements}
We thank Persi Diaconis for suggesting the problem and Jimmy He for the idea behind Lemma~\ref{lem:TotalScore}. We also thank two anonymous referees for whose comments greatly improved the presentation of the manuscript.
\bibliographystyle{alpha}
\bibliography{reference}

\end{document}